\documentclass[reqno, 12pt]{amsart}
\pdfoutput=1
\makeatletter
\let\origsection=\section \def\section{\@ifstar{\origsection*}{\mysection}} 
\def\mysection{\@startsection{section}{1}\z@{.7\linespacing\@plus\linespacing}{.5\linespacing}{\normalfont\scshape\centering\S}}
\makeatother        

\usepackage{amsmath,amssymb,amsthm}
\usepackage{mathrsfs}
\usepackage{mathabx}\changenotsign
\usepackage{dsfont}

\usepackage{verbatim}
 
\usepackage{graphicx}
\usepackage{xcolor}
\usepackage{tikz}
\usepackage{caption}
\usepackage{subcaption}

\usetikzlibrary{arrows.meta}

\usepackage[backref]{hyperref}
\usepackage[nameinlink, capitalise, noabbrev]{cleveref}
\hypersetup{
    colorlinks,
    linkcolor={red!60!black},
    citecolor={green!60!black},
    urlcolor={blue!60!black},
    pdftitle={An end degree for digraphs},
    pdfauthor={Matthias Hamann, Karl Heuer}
}

\usepackage[open,openlevel=2,atend]{bookmark}

\usepackage[abbrev,msc-links,backrefs]{amsrefs} 
\usepackage{doi}

\renewcommand{\PrintDOI}[1]{\doi{#1}}

\usepackage[T1]{fontenc}
\usepackage{lmodern}
\usepackage[babel]{microtype}
\usepackage[english]{babel}

\usepackage{geometry}
\numberwithin{equation}{section}
\numberwithin{figure}{section}

\usepackage[shortlabels]{enumitem}

\let\polishlcross=\l
\def\l{\ifmmode\ell\else\polishlcross\fi}

\def\paragraph#1{%
  \noindent\textbf{#1.}\enspace}

\let\emptyset=\varnothing
\let\setminus=\smallsetminus

\let\sm=\setminus

\makeatletter
\def\moverlay{\mathpalette\mov@rlay}
\def\mov@rlay#1#2{\leavevmode\vtop{   \baselineskip\z@skip \lineskiplimit-\maxdimen
   \ialign{\hfil$\m@th#1##$\hfil\cr#2\crcr}}}
\newcommand{\charfusion}[3][\mathord]{
    #1{\ifx#1\mathop\vphantom{#2}\fi
        \mathpalette\mov@rlay{#2\cr#3}
      }
    \ifx#1\mathop\expandafter\displaylimits\fi}
\makeatother

\DeclareFontFamily{U}  {MnSymbolC}{}
\DeclareSymbolFont{MnSyC}         {U}  {MnSymbolC}{m}{n}
\DeclareFontShape{U}{MnSymbolC}{m}{n}{
    <-6>  MnSymbolC5
   <6-7>  MnSymbolC6
   <7-8>  MnSymbolC7
   <8-9>  MnSymbolC8
   <9-10> MnSymbolC9
  <10-12> MnSymbolC10
  <12->   MnSymbolC12}{}
\DeclareMathSymbol{\powerset}{\mathord}{MnSyC}{180}

\let\epsilon=\varepsilon

\theoremstyle{plain}
\newtheorem{thm}{Theorem}[section]

\newtheorem{lemma}[thm]{Lemma}

\newtheorem{proposition}[thm]{Proposition}
\newtheorem{problem}[thm]{Problem}

\newtheorem*{claim*}{Claim}
\newtheorem{claim}{Claim}[]
\newtheorem{thm-intro}{Theorem}[]
\newtheorem{conj-intro}[thm-intro]{Conjecture}
\newtheorem{question-intro}[thm-intro]{Question}

\theoremstyle{definition}

\newtheorem{example}[thm]{Example}
\newtheorem*{example*}{Example}

\newcommand{\cP}{\mathcal P}
\newcommand{\dom}{\textnormal{dom}}

\newenvironment{claimproof}
{\begin{proof}
 [Proof.]
 \vspace{-1.5\parsep}
}
{\renewcommand{\qed}{\hfill $\Diamond$} \end{proof}}

\makeatletter
\newcommand\thankssymb[1]{\textsuperscript{\@fnsymbol{#1}}}

\makeatother

\begin{document}

\author[M.~Hamann]{Matthias Hamann\thankssymb{2}}
\address{Matthias Hamann, University of Hamburg, Department of Mathematics, Bundesstr. 55, 20146 Hamburg, Germany}
\email{\tt matthias.hamann@math.uni-hamburg.de}
\thanks{\thankssymb{2} Funded by the Deutsche Forschungsgemeinschaft (DFG) - Project No.\ 549406527.}

\author[K.~Heuer]{Karl Heuer}
\address{Karl Heuer, Technical University of Denmark, Department of Applied Mathematics and Computer Science, Richard Petersens Plads, Building 322, 2800 Kongens Lyngby, Denmark}
\email{\tt karheu@dtu.dk}

\title[]{An end degree for digraphs}

\date\today

\keywords{infinite graphs, digraphs, ends, end degrees, defining sequences}

\subjclass[2020]{05C63, 05C20}

\begin{abstract}
In this paper we define a degree for ends of infinite digraphs.
The well-definedness of our definition in particular resolves a problem by Zuther.
Furthermore, we extend our notion of end degree to also respect, among others, the vertices dominating the end, which we denote as combined end degree.
Our main result is a characterisation of the combined end degree in terms of certain sequences of vertices, which we call end-exhausting sequences.
This establishes a similar, although more complex relationship as known for the combined end degree and end-defining sequences in undirected graphs.
\end{abstract}

\maketitle

\section{Introduction}
\label{sec:intro}
The notion of ends became crucial for analysing the structure of infinite graphs.
An \textit{end} of a graph is an equivalence class of one-way infinite paths, where two such paths are called \textit{equivalent} if they are joined by infinitely many disjoint paths.
Degree parameters were defined for ends as well, see e.g.~\cites{Bruhn_Stein_end_degree, S2011}, where the basic definition is as follows.
The \textit{degree} of an end $\omega$ of a graph is the supremum (within $\mathbb{N} \cup \{ \infty \}$) of the number of disjoint one-way infinite paths in $\omega$.
It is a non-trivial theorem by Halin~\cite{Halin_end-degree} that the supremum in this definition is actually an attained maximum.
End degrees turned out to be useful parameters for infinite graphs, e.g. for characterising a topological notion of infinite cycles~\cite{Bruhn_Stein_end_degree}, or when studying extremal questions regarding the existence of infinite grid-like subgraphs~\cite{S2011}.

A different way to describe the degree of an end is by certain sequences of nested finite vertex separators, so-called \textit{defining sequences}.
It was shown in~\cite{Gollin_Heuer_combined_end_deg} that one can characterise the degree of an end together with the number of vertices dominating it, also referred to as \textit{combined end degree}, via the sizes of the separators within defining sequences.
Here, a vertex $v$ is said to \textit{dominate} an end $\omega$ if there exist infinitely many paths from $v$ to a one-way infinite path in~$\omega$ which are all disjoint except from~$v$.

\medskip

In this paper we shall consider directed graphs, briefly denoted as \emph{digraphs}, which are infinite, and we shall define analogous concepts of end degrees as mentioned above for undirected graphs.
For this we follow a notion of ends of digraphs defined by Zuther~\cites{Z1997,Z1998}, which is a natural and analog definition to the one for undirected graphs.
An \textit{end} of a digraph is defined as an equivalence class of rays and anti-rays, where a (\textit{anti-})\textit{ray} is an orientation of a one-way infinite path such that each edge is oriented towards (resp.~away from) infinity (see Section~\ref{sec:prelims} for a precise definition).
Zuther called two rays or anti-rays $R_1, R_2$ of a digraph $D$ \textit{equivalent} if there exist infinitely pairwise disjoint directed paths from $R_1$ to $R_2$ and vice versa.
Note that this definition allows that e.g. $R_1$ is a ray and $R_2$ is an anti-ray.

The first result of this paper, Theorem~\ref{thm:endDeg}, resolves a problem stated by Zuther~\cite{Z1997}*{Problem~2} and proves that an end of a digraph which contains any finite number of disjoint rays also contains infinitely many disjoint rays.
Hence, Theorem~\ref{thm:endDeg} is an analogous result to the aforementioned theorem by Halin~\cite{Halin_end-degree}, and allows us to define the \textit{in-degree} (resp.~\textit{out-degree}) of an end as the maximum number of disjoint rays (resp.~anti-rays) in that end.

The natural question arises whether an end with infinite in- and out-degree might admit a system of pairwise disjoint rays and anti-rays witnessing both of these degrees.
By Theorem~\ref{thm:InfRaysInfAntiRays} we answer this question negatively and construct a digraph with infinite in- and out-degree where each ray intersects each anti-ray.

The main contribution of this paper is the introduction of \textit{end-exhausting sequences}, a concept for ends of digraphs similar to defining sequences for ends of undirected graphs.
Similarly, although more complex as for undirected graphs, we define a \textit{combined in-degree} (and \textit{out-degree}) for ends, and prove an equality to a parameter solely based on end-exhausting sequences in the following main result of this paper:

\begin{thm}\label{thm:main}
Let $D$ be a digraph and let $\omega$ be an end of~$D$ that contains at least one but at most countably many rays.
Then the combined in-degree of~$\omega$ is the same as
\[
\inf\left\{\liminf_{i\in\mathbb N}|U_i|\Bigm|(U_i)_{i\in\mathbb N} \text{ is an }\omega\text{-exhausting sequence}\right\},
\]
where both values are considered in~$\mathbb N\cup\{\infty\}$.
\end{thm}

Qualitatively, Theorem~\ref{thm:main} establishes the same duality between end-exhausting sequences and combined in-degrees (or out-degrees) of ends for digraphs as it is known for end-defining sequences and the combined end degree in the undirected case.

The structure of this paper is as follows. 
After introducing some terminology in Section~\ref{sec:prelims}, we prove in Section~\ref{sec:deg} that the in- and out-degree of an end is well-defined.
In Section~\ref{sec:ex} we construct a digraph containing infinitely many disjoint rays and infinitely many disjoint anti-rays such that each ray intersects each anti-ray.
In Section~\ref{sec:endsequence} we define end-exhausting sequences and the combined in- and out-degree of ends, followed by the proof of Theorem~\ref{thm:main}.
Finally, we briefly discuss in Section~\ref{sec:edge deg} how the results from Section~\ref{sec:deg} and Section~\ref{sec:ex} can be proved when edge-disjoint rays (and anti-rays) are considered instead of vertex-disjoint ones.

\section{Preliminaries}
\label{sec:prelims}

For general facts and notation regarding graphs we refer the reader to~\cite{diestel}, regarding digraphs in particular to~\cite{bang-jensen}.

We call a digraph $D$ \textit{weak} if it is weakly connected.
For the sake of brevity we call a directed cycle just a \textit{dicycle} and a directed path just a \textit{dipath}.
Given a dipath $P$ containing two vertices $a, b$ such that $b$ is reached from $a$ via~$P$, we define $aPb$ as the subdipath of $P$ starting at $a$ and ending in $b$.
Given two vertex sets $A$, $B$, we call a dipath $P$ an \emph{$A$--$B$ dipath} if $P$ starts in $A$, ends in $B$ and is internally disjoint from $A \cup B$.

We call a weak digraph where each vertex has in- and out-degree $1$ except one vertex $v$ which has in-degree (resp.\ out-degree) $0$ and out-degree (resp.\ in-degree) $1$ a \textit{ray} (resp.\ \textit{anti-ray}).
The vertex $v$ is called the \textit{starting vertex} (resp.\ \textit{end vertex}) of the ray (resp.\ anti-ray).
We say that a ray \textit{starts} in a vertex set $A$ if it has its starting vertex in~$A$.
Given a ray $R$ with starting vertex $v$ and some $x \in V(R)$ we denote by $Rx$ the subdipath $vRx$ of $R$.
A \emph{tail} of~$R$ is a subray of~$R$.
If this tail starts at~$x$, then we denote it by~$xR$.
Similarly, for an anti-ray $Q$ with end vertex~$v$ and some $x\in V(Q)$, we denote by $xQ$ the subdipath $xQv$ of~$Q$ and by $Qx$ the subanti-ray of~$Q$ that ends at~$x$, which we will also call a \emph{tail} of~$Q$.

Let $Q$ and~$R$ be rays or anti-rays.
We write $Q\leq R$ if there are infinitely many pairwise disjoint $Q$--$R$ dipaths and we write $Q\sim R$ if $Q\leq R$ and $R\leq Q$.
Then $\leq$ is a partial order on the set of rays and anti-rays in a digraph~$D$ and $\sim$ is an equivalence relation on that set.
The equivalence classes of~$\sim$ are the \emph{ends} of~$D$ and we can extend the relation $\leq$ to the ends: we write $\eta\leq\omega$ for ends $\eta$ and $\omega$ if there are $Q\in\eta$ and $R\in\omega$ with $Q\leq R$.
Note that $\eta\leq\omega$ if and only if $Q\leq R$ for every $Q\in\eta$ and $R\in\omega$.
In particular, we have $\eta\leq\omega$ and $\omega\leq\eta$ if and only if $\eta=\omega$.

We call an oriented tree $A$ that contains a vertex $x$ such that each vertex $y \in V(A)$ is reachable from $x$ (resp.~reaches $x$) in $A$ by a dipath, an \textit{out-arborescence} (resp.~\textit{in-arborescence}).
The vertex $x$ is called the \textit{root} of $A$.
An out-arborescence (resp.~in-arborescence) $S$ with root $c$ whose underlying tree is a star is called an \textit{out-star} (resp.~\textit{in-star}) with \textit{centre} $c$.

An undirected tree $C$ is called a \textit{comb} if it is obtained from a system $\cP$ of infinitely many pairwise disjoint finite paths and a one-way infinite path $R$ by gluing one end vertex of each path to a vertex on $R$ such that different paths are glued to different vertices on $R$.
The one-way infinite path $R$ naturally exists as a subgraph in $C$ and is called the \textit{spine} of $C$.
All those end vertices of paths in~$\cP$ that do not lie on~$R$ and those that belong to trivial paths (i.e.~paths that consist of just one vertex) are called the \textit{teeth} of $C$.
An orientation of a comb is called an \textit{out-comb} (resp.~\textit{in-comb}) if the spine of the comb is oriented as a ray (resp.~anti-ray) and each path of $\cP$ is oriented as a dipath directed away from (resp.~towards) $R$.

We shall need the following analog for digraphs of the so-called Star-Comb Lemma~\cite{diestel}*{Lemma~8.2.2} for undirected graphs.
The proof is very similar to the undirected version, but since it is short, we include it for the sake of completeness.

\begin{lemma}[Star-Comb Lemma]\label{lem:StarComb}
    Let $D$ be a digraph and let $x\in V(D)$ and $U\subseteq V(D)$ be infinite such that there exists an $x$--$u$ dipath for every $u\in U$.
    Then there exists either an out-comb with all its teeth in~$U$ or a subdivided infinite out-star with all its leaves in~$U$.
\end{lemma}

\begin{proof}
    By Zorn's lemma, there exists a maximal out-arborescence $T$ containing vertices of~$U$ such that for every vertex $t$ of~$V(T)$ there exists an $x$--$t$ dipath in~$T$ and such that every vertex without out-neighbour lies in~$U$.
    Since $U$ is infinite, $T$ is infinite as well.
    If $T$ has a vertex of infinite out-degree, then it contains a subdivided infinite out-star as out-arborescence with its centre as root and its leaves in~$U$.
    So let us assume that all vertices of~$T$ have finite out-degree.
    Then there exists a ray in~$T$ starting at~$x$, which we denote by $R$.
    In order to construct infinitely many pairwise disjoint $R$--$U$ dipaths in~$T$, let us assume that we have already constructed $P_0,\ldots, P_{n-1}$ such that the starting vertex of $P_i$ lies before that of~$P_j$ on~$R$ for $i<j$.
    Let $v$ be the starting vertex of~$P_{n-1}$ and let $w$ be its out-neighbour on~$R$.
    Then the edge $vw$ lies on an $x$--$u$ dipath in~$T$ for some $u\in U$.
    This dipath contains a maximal $R$--$U$ dipath~$P_n$, which is disjoint to all $P_i$ with $i < n$.
    Thus, we obtain infinitely many pairwise disjoint $R$--$U$ dipaths.
    Then $R$ together with all dipaths $P_i$ for $i\in\mathbb N$ is an out-arborescence with root~$x$, that is an out-comb.
\end{proof}

\section{End degree}
\label{sec:deg}

Zuther \cite{Z1997}*{Theorem 2.17}, see also Gut et al.\ \cite{GKR2024}, proved that every digraph that contains an arbitrarily large finite number of pairwise disjoint rays contains infinitely many pairwise disjoint rays.
Zuther posed the problem \cite{Z1997}*{Problem 2} whether this also holds when we ask all rays to lie in a common end.
We settle his problem in the positive.

\begin{thm}\label{thm:endDeg}
Let $D$ be a digraph.
\begin{enumerate}[label = {\rm (\roman*)}]
    \item\label{itm:endDeg1} If an end of~$D$ contains $n$ pairwise disjoint rays for all $n\in\mathbb N$, then it contains infinitely many pairwise disjoint rays.
    \item\label{itm:endDeg2} If an end of~$D$ contains $n$ pairwise disjoint anti-rays for all $n\in\mathbb N$, then it contains infinitely many pairwise disjoint anti-rays.
\end{enumerate}
\end{thm}

\begin{proof}
It suffices to prove \ref{itm:endDeg1}, since \ref{itm:endDeg2} follows by applying \ref{itm:endDeg1} to the digraph with all edge directions reversed.

Let $\omega$ be an end of a fixed digraph $D$ such that for all $n\in\mathbb N$ there are $n$ pairwise disjoint rays in~$\omega$ and let $R$ be a ray in~$\omega$.
For all $n\in\mathbb N$, we will recursively construct a set $\mathcal R^n=\{R_1^n,\ldots,R_n^n\}$ of $n$ pairwise disjoint rays, a set $X^n:=\{x_1^n,\ldots, x_n^n\}$ of $n$ vertices, and a set $\mathcal P^n$ of $2n$ dipaths, such that the following hold for all $n\geq 1$:
\begin{enumerate}[label = (\arabic*)]
    \item\label{itm:pfendDeg1} $\mathcal R^n\subseteq\omega$;
    \item\label{itm:pfendDeg3} $x_i^n$ lies on $R_i^n$ for all $1\leq i\leq n$;
    \item\label{itm:pfendDeg4} $R_i^{n-1}x_i^{n-1}$ is a proper starting subdipath of $R_i^nx_i^n$ for all $1\leq i\leq n-1$;
    \item\label{itm:pfendDeg5} $\mathcal P^n$ contains a dipath from $R$ to $x_i^{n-1}R_i^nx_i^n$ and a dipath from $x_i^{n-1}R_i^nx_i^n$ to~$R$ for all $1\leq i\leq n-1$ that avoid $\bigcup_{j<n}\mathcal P^j$, where $x_i^{0}$ denotes the starting vertex of $R_i$.
\end{enumerate}
Let $\mathcal R^1$ be a set consisting of a single ray $R_1^1$ in~$\omega$, let $\mathcal P^1$ consist of an $R$--$R_1^1$ dipath and an $R_1^1$--$R$ dipath and let $X^1$ consist of a vertex of~$R_1^1$ that lies after all vertices of dipaths in~$\mathcal P^1$ on~$R_1^1$.
By these definitions, \ref{itm:pfendDeg1}--\ref{itm:pfendDeg5} are satisfied for $n=1$.
Let us now assume that we have already constructed $\mathcal R^n$, $X^n$ and~$\mathcal P^n$.

Let $X$ be the set of vertices on the dipaths $R_i^nx_i^n$ for $1\leq i\leq n$ and let $\mathcal Q\subseteq\omega$ be a set of $n+1$ pairwise disjoint rays.
Let $Q_1,\ldots, Q_{n+1}$ be tails of the elements of~$\mathcal Q$ that avoid~$X$.
For all $1\leq i\leq n$ and $1\leq j\leq n+1$, let $P_{i,j}^1,\ldots,P_{i,j}^n$ be $n$ pairwise disjoint $x_i^nR_i^n$--$Q_j$ dipaths that avoid $X \sm X^n$, which exist since all considered rays lie in a common end.
For all $1\leq i \leq n$ let $h_i$ denote the last vertex on $R^n_{i}$ that lies on any of the dipaths $P_{i,j}^k$.
Now, for all $1\leq \ell \leq n+1$, let $y_{\ell}$ be a vertex on~$Q_{\ell}$ that lies after all vertices on dipaths $P_{i,j}^k$ on~$Q_{\ell}$ and after all vertices on segments $x^n_iR^n_ih_i$.

Let $D'$ be the finite subdigraph of~$D$ induced by all dipaths $x_i^nR_i^nz$, where $z$ is a starting vertex of some $P_{i,j}^k$, by all dipaths $P_{i,j}^k$, and by all dipaths $z'Q_jy_j$, where $z'$ is an end vertex of some dipath $P_{i,j}^k$.
Let $S$ be a set of less than $n$ vertices in~$D'$.
Then $S$ avoids at least one ray $x_i^nR_i^n$, at least one $Q_j$ and at least one $P_{i,j}^k$, that is, we find an $x_i^n$--$y_j$ dipath.
Menger's theorem implies that there are $n$ disjoint dipaths from $X^n$ to $\{y_1,\ldots,y_{n+1}\}$.
We may assume that the indices are such that we find dipaths $P_i$ from $x_i^n$ to~$y_i$ for all $1\leq i\leq n$.
We set $R_i^{n+1}:=R_i^nx_i^nP_iy_iQ_i$ for all $1\leq i\leq n$ and choose $R_{n+1}^{n+1}$ as a tail of $Q_{n+1}$ which is disjoint from all $R_i^{n+1}$.
Finally, we set $\mathcal R^{n+1}:=\{R_i^{n+1}\mid 1\leq i\leq n+1\}$.
By construction, \ref{itm:pfendDeg1} holds for $\mathcal R^{n+1}$.
Let $\mathcal{P}^{n+1}$ be a set of dipaths, one from $R$ to $R_i^{n+1}$ and one from $R_i^{n+1}$ to~$R$ for all $1\leq i\leq n$, such that all these dipaths avoid the vertices in $\bigcup_{j\leq n}\mathcal P^j$.
Note that these dipaths exist as all rays lie in a common end.
For $1\leq i\leq n+1$, let $x_i^{n+1}$ be a vertex on $R_i^{n+1}$ after all vertices of elements of $\mathcal P^{n+1}$.
In particular, $x_i^{n+1}$ lies after $x_i^n$ on $R_i^{n+1}$.
Then, \ref{itm:pfendDeg3}--\ref{itm:pfendDeg5} hold by construction.

Thus $\{\bigcup R_i^nx_i^n\mid n\in\mathbb N, 1\leq i\leq n\}$ is an infinite set of pairwise disjoint rays that are all equivalent to~$R$, and hence lie in~$\omega$.
\end{proof}

The \emph{in-degree} of an end $\omega$, denoted by $d^-(\omega)$, is the maximum number of pairwise disjoint rays in that end, if finite, and $\infty$ otherwise.
Analogously, we define the \emph{out-degree}, denoted by $d^+(\omega)$, with respect to anti-rays.

\section{Example}
\label{sec:ex}

In this section, we will discuss an example of a digraph with infinitely many pairwise disjoint rays and infinitely many pairwise disjoint anti-rays such that every ray and every anti-ray share a vertex.

\begin{thm}\label{thm:InfRaysInfAntiRays}
There exists a digraph $D$ with the following properties:
\begin{enumerate}[label = \rm (\roman*)]
    \item\label{itm:example1} $D$ contains infinitely many pairwise disjoint rays.
    \item\label{itm:example2} $D$ contains infinitely many pairwise disjoint anti-rays.
    \item\label{itm:example3} Every ray and every anti-ray of~$D$ share a vertex.
\end{enumerate}
\end{thm}

\begin{proof}
For all $i\in\mathbb N$ and $i':=(\sum^{i-1}_{j=1}j)+1$, let $R_i=x^i_{i'}x^i_{i'+1}\ldots$ be a ray.
Let $D$ be the union of all rays~$R_i$ with additional edges $x^{i+1}_{k}x^i_k$ for all $i\in\mathbb N$ and all feasible~$k$ and with additional edges from $x^1_k$ with $k=(\sum^{i}_{j=1}j)+k'$ to $x^i_{i''}$ with $i''=(\sum^{i-1}_{j=1}j)+k'$, for every $k' \in \{1, \ldots, i\}$.
We call these latter edges \emph{diagonal}.
Note that this digraph is planar.
See Figure~\ref{fig:counterexample}~\subref{subfig:counter} for one picture of $D$ and Figure~\ref{fig:planer_counterexample}~\subref{subfig:planar_counter} for a planar drawing of~$D$.

\begin{figure}[ht]
    \centering
       \begin{subfigure}[b]{0.48\linewidth}
    \centering
\begin{tikzpicture}
    \draw[ -{Latex[length=2.75mm]}] (0,0.5) -- (0,5.25);
    \foreach \y in {0.5,1,1.5,2,2.5,3,3.5,4,4.5} {
        \fill (0,\y) circle (2pt);
    }
    
    \draw (0,-0.25) node[anchor=south] {$R_1$};

    \draw[ -{Latex[length=2.75mm]}] (1,1) -- (1,5.25);
    \foreach \y in {2,2.5,3,3.5,4,4.5} {
        \fill (1,\y) circle (2pt);
    }
        \fill (1,1) circle (2pt);
        \fill (1,1.5) circle (2pt);

    \draw (1,0.25) node[anchor=south] {$R_2$};

    \draw[ -{Latex[length=2.75mm]}] (2,2) -- (2,5.25);
    \foreach \y in {2,2.5,3,3.5,4,4.5} {
        \fill (2,\y) circle (2pt);
    }

    \draw (2,1.25) node[anchor=south] {$R_3$};

\draw[blue, -{Latex[length=2.75mm]}] (1,1) -- (0,1);
\draw[blue, -{Latex[length=2.75mm]}] (2,2) -- (1,2);
\draw[blue, -{Latex[length=2.75mm]}] (1,2) -- (0,2);
\draw[blue, -{Latex[length=2.75mm]}] (0,2) -- (1,1);
\draw[blue, -{Latex[length=2.75mm]}] (0,3.5) -- (2,2);
\draw[blue, -{Latex[length=2.75mm]}] (2,3.5) -- (1,3.5);
\draw[blue, -{Latex[length=2.75mm]}] (1,3.5) -- (0,3.5);

\draw[blue, -{Latex[length=2.75mm]}] (2.85,3.5) -- (2,3.5);

\draw[red, -{Latex[length=2.75mm]}] (1,1.5) -- (0,1.5);
\draw[red, -{Latex[length=2.75mm]}] (2,2.5) -- (1,2.5);
\draw[red, -{Latex[length=2.75mm]}] (1,2.5) -- (0,2.5);
\draw[red, -{Latex[length=2.75mm]}] (0,2.5) -- (1,1.5);
\draw[red, -{Latex[length=2.75mm]}] (0,4) -- (2,2.5);
\draw[red, -{Latex[length=2.75mm]}] (2,4) -- (1,4);
\draw[red, -{Latex[length=2.75mm]}] (1,4) -- (0,4);

\draw[red, -{Latex[length=2.75mm]}] (2.85,4) -- (2,4);

\draw[orange, -{Latex[length=2.75mm]}] (2,3) -- (1,3);
\draw[orange, -{Latex[length=2.75mm]}] (1,3) -- (0,3);
\draw[orange, -{Latex[length=2.75mm]}] (0,4.5) -- (2,3);
\draw[orange, -{Latex[length=2.75mm]}] (2,4.5) -- (1,4.5);
\draw[orange, -{Latex[length=2.75mm]}] (1,4.5) -- (0,4.5);

\draw[orange, -{Latex[length=2.75mm]}] (2.85,4.5) -- (2,4.5);

    \draw (0.5,5.25);
    \fill (0.5,5.25) circle (1pt);
    \draw (0.5,5.5);
    \fill (0.5,5.5) circle (1pt);
    \draw (0.5,5.75);
    \fill (0.5,5.75) circle (1pt);    

    \draw (1.5,5.25);
    \fill (1.5,5.25) circle (1pt);
    \draw (1.5,5.5);
    \fill (1.5,5.5) circle (1pt);
    \draw (1.5,5.75);
    \fill (1.5,5.75) circle (1pt);    

    \draw (2.85,3.75);
    \fill (2.85,3.75) circle (1pt);
    \draw (3.1,3.75);
    \fill (3.1,3.75) circle (1pt);
    \draw (3.35,3.75);
    \fill (3.35,3.75) circle (1pt);

    \draw (2.675,5.175);
    \fill (2.675,5.175) circle (1pt);
    \draw (2.85,5.35);
    \fill (2.85,5.35) circle (1pt);       
    \draw (3.025,5.525);
    \fill (3.025,5.525) circle (1pt);       
     
\end{tikzpicture}
    \caption{The digraph $D$.}
\label{subfig:counter}
        \end{subfigure}
        \begin{subfigure}[b]{0.48\linewidth}
    \centering

\begin{tikzpicture}

\pgfmathsetmacro{\x}{cos(45)}
\pgfmathsetmacro{\y}{sin(45)}

\pgfmathsetmacro{\a}{cos(-5)}
\pgfmathsetmacro{\b}{sin(-5)}

\pgfmathsetmacro{\c}{cos(-10)}
\pgfmathsetmacro{\d}{sin(-10)}

\draw[blue, -{Latex[length=2.75mm]}] ({\x},{\y}) arc[start angle=45, end angle=90, radius=1];
\draw[blue, -{Latex[length=2.75mm]}] (2,0) arc[start angle=0, end angle=45, radius=2];
\draw[blue, -{Latex[length=2.75mm]}] ({2*\x},{2*\y}) arc[start angle=45, end angle=90, radius=2];

\draw[blue, -{Latex[length=2.75mm]}] (3.5,0) arc[start angle=0, end angle=45, radius=3.5];
\draw[blue, -{Latex[length=2.75mm]}] ({3.5*\x},{3.5*\y}) arc[start angle=45, end angle=90, radius=3.5];

\draw[blue, -{Latex[length=2.75mm]}] ({3.5*\c},{3.5*\d}) arc[start angle=-10, end angle=0, radius=3.5];

\draw[blue, rounded corners=2.5mm, -{Latex[length=2.75mm]}] (0,2) -- (-0.5,2) -- (-0.5,0) -- ({\x},0) -- ({\x},{\y});

\draw[blue, rounded corners=2.5mm, -{Latex[length=2.75mm]}] (0,3.5) -- (-1,3.5) -- (-1,-0.5) -- (2,-0.5) -- (2,0);

\draw[red, -{Latex[length=2.75mm]}] ({1.5*\x},{1.5*\y}) arc[start angle=45, end angle=90, radius=1.5];
\draw[red, -{Latex[length=2.75mm]}] (2.5,0) arc[start angle=0, end angle=45, radius=2.5];
\draw[red, -{Latex[length=2.75mm]}] ({2.5*\x},{2.5*\y}) arc[start angle=45, end angle=90, radius=2.5];

\draw[red, -{Latex[length=2.75mm]}] (4,0) arc[start angle=0, end angle=45, radius=4];
\draw[red, -{Latex[length=2.75mm]}] ({4*\x},{4*\y}) arc[start angle=45, end angle=90, radius=4];

\draw[red, -{Latex[length=2.75mm]}] ({4*\c},{4*\d}) arc[start angle=-10, end angle=0, radius=4];

\draw[red, rounded corners=2.5mm, -{Latex[length=2.75mm]}] (0,2.5) -- (-0.75,2.5) -- (-0.75,-0.25) -- ({1.5*\x},-0.25) -- ({1.5*\x},{1.5*\y});

\draw[red, rounded corners=2.5mm, -{Latex[length=2.75mm]}] (0,4) -- (-1.25,4) -- (-1.25,-0.75) -- (2.5,-0.75) -- (2.5,0);

\draw[orange, -{Latex[length=2.75mm]}] (3,0) arc[start angle=0, end angle=45, radius=3];
\draw[orange, -{Latex[length=2.75mm]}] ({3*\x},{3*\y}) arc[start angle=45, end angle=90, radius=3];

\draw[orange, -{Latex[length=2.75mm]}] (4.5,0) arc[start angle=0, end angle=45, radius=4.5];
\draw[orange, -{Latex[length=2.75mm]}] ({4.5*\x},{4.5*\y}) arc[start angle=45, end angle=90, radius=4.5];

\draw[orange, -{Latex[length=2.75mm]}] ({4.5*\c},{4.5*\d}) arc[start angle=-10, end angle=0, radius=4.5];

\draw[orange, rounded corners=2.55mm, -{Latex[length=2.75mm]}] (0,4.5) -- (-1.5,4.5) -- (-1.5,-1) -- (3,-1) -- (3,0);

    \draw[-{Latex[length=2.75mm]}] (0,0.5) -- (0,5.25);
    \draw (0,5.25) node[anchor=south] {$R_1$};

    \draw[-{Latex[length=2.75mm]}] ({\x},{\y}) -- ({5.25*\x},{5.25*\y});
    \draw ({5.25*\x+0.2},{5.25*\y}) node[anchor=south] {$R_2$};

    \draw[-{Latex[length=2.75mm]}] (2,0) -- (5.25,0);
    \draw (5.25,0) node[anchor=west] {$R_3$};

    \foreach \y in {0.5,1,1.5,2,2.5,3,3.5,4,4.5} {
    \fill (0,\y) circle (2pt);
    }

    \foreach \z in {1,1.5,2,2.5,3,3.5,4,4.5} {
    \fill ({\z*\x},{\z*\y}) circle (2pt);
    }

    \foreach \y in {2,2.5,3,3.5,4,4.5} {
    \fill (\y,0) circle (2pt);
    }

    \draw (3.75,-0.9);
    \fill (3.75,-0.9) circle (1pt);
    \draw (4,-1.025);
    \fill (4,-1.025) circle (1pt);
    \draw (4.25,-1.15);
    \fill (4.25,-1.15) circle (1pt);

\end{tikzpicture}

        \caption{A planar drawing of the digraph $D$.}
 \label{subfig:planar_counter}

        \end{subfigure}
\caption{Two drawings of the digraph $D$.}
\label{fig:counterexample}
\label{fig:planer_counterexample}
\end{figure}
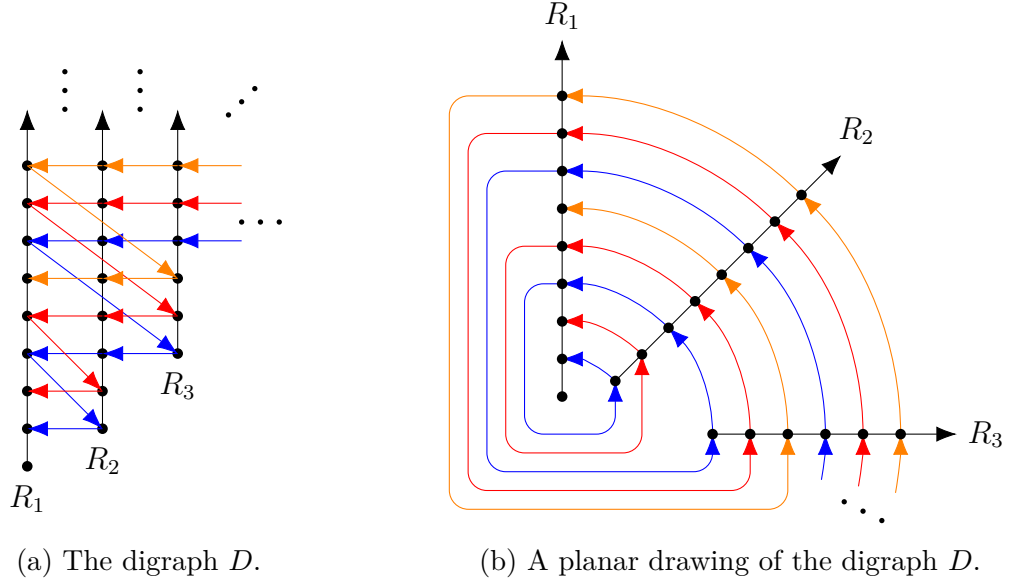

We note that~\ref{itm:example1} and~\ref{itm:example2} are trivially true.
So let $Q$ be a ray and $P$ be an anti-ray in~$D$.
Moving on $P$ along the edges in opposite direction away from the end vertex of~$P$, we must meet $R_1$ after finitely many edges of the form $x_n^{m+1}x_n^m$ or $x_n^mx_{n+1}^m$.
This implies that $P$ meets $R_1$ infinitely many times.
Thus, $P$ must also meet $R_2$ infinitely often, and so on.
So $P$ meets each $R_i$ infinitely often.
Thus, if $Q$ and some $R_i$ have a common tail, then $Q$ and $P$ must have infinitely many common vertices.

So let us assume that $Q$ has no common tail with any~$R_i$.
Note that $Q$ must leave $R_i$ towards $R_{i-1}$, and $R_{i-1}$ towards $R_{i-2}$, and so on until it meets $R_1$.
Then $Q$ must use a diagonal edge to some $R_k$ with $k \geq i$ when leaving $R_1$.
Afterwards, $Q$ must again traverse all $R_{k'}$ with $k'\leq k$.
Note that for each ray $R_i$ there are only finitely diagonal edges that are directed towards that ray.
So eventually, $Q$ must use a diagonal edge to some $R_k$ with $k > i$ when leaving $R_1$.
Hence, we obtain that $Q$ meets all $R_i$ infinitely often.
Note, furthermore, that due to planarity it is not possible for~$Q$ to traverse a ray $R_i$ first through some vertex $r^i_n$ and later through some vertex $r^i_m$ for $m < n$ as the $r^i_n$--$R_1$ subdipath of~$Q$ together with the $R_1$--$r^i_m$ subdipath of~$Q$ would cause $r^i_mQ$ to intersect $Qr^i_m$ in another vertex than $r^i_m$, which is impossible.

This implies that there exists $x_{\ell_1}^j$ and $x_{\ell_2}^j$ on some~$R_j$ and on~$Q$ with $\ell_1<\ell_2$ such that $x_{\ell_1}^jR_jx_{\ell_2}^j$ does not meet~$Q$ and that there exists $x_\ell^j$ on~$P$ such that $\ell < \ell_1$, but there is no $x_{\ell^*}^j$ on~$P$ with $\ell < \ell^* < \ell_1$.
Hence, $P$ must use a diagonal edge $x^1_{k}x_{\ell'}^{j'}$ after the vertex $x_\ell^j$ with ${\ell'} \leq {\ell}$ and $j' \geq j$ for some $k \in \mathbb{N}$.
Since $\ell < \ell_1$, the dipath $x_{\ell_1}^jQx_{\ell_2}^j$ must also use a diagonal edge $x^1_{k'}x_{\ell''}^{j''}$ with $k' > k$ for some $j'', {\ell''} \in \mathbb{N}$.
Now it follows that $x_\ell^j$ lies in the interior of the face bounded by $x_{\ell_1}^jQx_{\ell_2}^j\cup x_{\ell_1}^jR_jx_{\ell_2}^j$.
Since $P$ intersects $R_j$ again after $x_{\ell}^j$, but uses the diagonal edge $x^1_{k}x_{\ell'}^{j'}$, we get that $P$ must intersect the subdipath $x_{\ell_1}^jQx^1_{k'}$ of $Q$, which is a contradiction.
\end{proof}

This leaves the following problem open.

\begin{problem}\label{prob:ubiquity:rays+anti}
If a digraph $D$ has an end $\omega$ such that $D$ contains for all $n\in\mathbb N$ a set of $n$ rays and $n$ anti-rays that are pairwise disjoint and lie in~$\omega$, does there exist a set of infinitely many rays and infinitely many anti-rays in~$\omega$ that are pairwise disjoint?    
\end{problem}

Note that the proof method used in Theorem~\ref{thm:endDeg} does not work here since we would need to find two disjoint dipath systems, one for rerouting our initial segments of rays, and one for our endsegments of anti-rays.
This, however, is not guaranteed by an application of Menger's theorem as done before.

The previous problem is also motivated by Proposition~\ref{prop:rays+anti=double} and the following problem by Gut et al.\ \cite{GKR2024}, in that a positive answer to Problem~\ref{prob:ubiquity:rays+anti} is a corollary of Proposition~\ref{prop:rays+anti=double} and of a positive answer for Problem~\ref{prob:GKR1.3}.

\begin{problem}\cite{GKR2024}*{Problem~1.3}\label{prob:GKR1.3}
    Is the consistently oriented double ray, i.\,e.\ the weak digraph where every vertex has in-degree and out-degree $1$, ubiquitous?\footnote{A digraph $H$ is \emph{ubiquitous} if, for any digraph $D$, the existence of $n$ pairwise disjoint copies of~$H$ in $D$ for all $n\in\mathbb N$ implies the existence of infinitely many pairwise disjoint copies of~$H$ in $D$.}
\end{problem}

\begin{proposition}\label{prop:rays+anti=double}
For every $n\in\mathbb N$, if a digraph $D$ contains a set of $n$ rays and $n$ anti-rays that are all pairwise disjoint and lie in the same end, then there exists a set of $n$ pairwise disjoint double rays in~$D$ all of whose tails lie in that end.
\end{proposition}

\begin{proof}
Let $R_1,\ldots, R_n$ be rays and $Q_1,\ldots, Q_n$ be anti-rays all in the same end $\omega$ of the digraph $D$ and all pairwise disjoint.
Let $\cP$ be a set of pairwise disjoint dipaths that consists of $n$ many $Q_i$--$R_j$ dipaths for all $1\leq i,j\leq n$.
This is possible to choose since all rays and anti-rays lie in $\omega$.
For all $1\leq i\leq n$, let $x_i$ be a vertex on $Q_i$ such that $Q_ix_i$ contains no vertex from any dipath in~$\cP$.
Let $y_i$ be a vertex on~$R_i$ such that $y_iR_i$ contains no vertex from any dipath in~$\cP$.
Let $H$ be the finite digraph on the final subdipaths $x_iQ_i$, the starting dipaths $R_iy_i$ and the dipaths in~$\cP$.
Then every set of less than $n$ vertices misses one dipath $x_iQ_i$, one dipath $R_jy_j$ and one $Q_i$--$R_j$ dipath $P\in\cP$.
Thus, this vertex set does not separate $X:=\{x_i\mid 1\leq i\leq n\}$ from $Y:=\{y_i\mid 1\leq i\leq n\}$.
By Menger's theorem, there exist $n$ pairwise disjoint $X$--$Y$ dipaths in~$H$ and hence in~$D$.
These dipaths together with the tails $Q_ix_i$ and $y_iR_i$ form $n$ pairwise disjoint double rays all of whose tails lie in~$\omega$.
\end{proof}

\section{End-exhausting sequences}
\label{sec:endsequence}
In this section we define a generalisation of the in-degree of an end, the so-called combined in-degree, and characterise it in terms of certain sequences of vertex sets.
Hence, we focus on ends that contain rays.
Everything can be done for the out-degree and anti-rays completely analogously, which is why we omit the details for that here.

Let $D$ be a digraph and $\omega$ an end of~$D$ which contains a ray.
Furthermore, let $(U_i)_{i\in\mathbb N}$ be a sequence of finite vertex sets of~$D$.
We say that the sequence $(U_i)_{i\in\mathbb N}$ is \emph{$\omega$-exhausting} if for every ray in~$\omega$ there exists an $i\in\mathbb N$ such that this ray contains a vertex of~$U_i$ and if a ray in~$\omega$ contains a vertex of~$U_i$, then it contains a vertex of~$U_{i+1}$.
Furthermore, note that obviously every countable digraph admits an $\omega$-exhausting sequence for every end $\omega$ that contains a ray of the digraph.
For uncountable digraphs this is not necessarily true.
The following proposition characterises the existence of exhausting sequences and is an analogue of Lemma~5.1 in~\cite{Gollin_Heuer_combined_end_deg}, which characterises the existence of so-called end-defining sequences for ends of undirected graphs.

\begin{proposition}\label{prop:exhaust_sequence_exist}
Let $D$ be a digraph and let $\omega$ be an end of~$D$ that contains a ray.
Then there exists an $\omega$-exhausting sequence if and only if there exist at most countably many disjoint rays in $\omega$.
\end{proposition}

\begin{proof}
If there are uncountably many disjoint rays in $\omega$, then clearly there cannot exist an $\omega$-exhausting sequence.

Conversely, let $\mathcal R:=\{R_i=x^i_0x^i_1\ldots \mid i \in \mathbb{N} \}$ be a maximal set of countably many disjoint rays in~$\omega$.
We set
\begin{align*}
    V_1&:=\{x^1_0\}\\
    V_2&:=V_1 \cup \{x^1_1,x^2_0\}\\
    V_3&:=V_2 \cup \{x^1_2,x^2_1,x^3_0\}\\
    \vdots
\end{align*}
Let $R$ be a ray in~$\omega$.
By the maximality of~$\mathcal R$, there exists some vertex from a ray in~$\mathcal R$ that lies on~$R$.
Let $x^i_j$ be the first such vertex on $R$.
Since $x^i_j$ lies in $V_{i+j}$ and thus in $V_k$ for all $k\geq i+j$, we conclude that $(V_i)_{i\in\mathbb N}$ is an $\omega$-exhausting sequence.
\end{proof}

Let us call an end $\omega$ of a digraph $D$ \emph{countable} if it contains at most countably many disjoint rays.
So by the proposition above, every countable end with at least one ray admits an exhausting sequence.

Let $X$ and~$Y$ be disjoint sets of ends of~$D$.
We say that a vertex set $S \subseteq V(D)$ \emph{separates} $X$ from~$Y$ in~$D$ if for every $\omega\in Y$ every $R\in\omega$ has a tail~$Q$ such that every ray in elements of~$X$ that starts at a vertex of~$Q$ meets~$S$.
In case $X$ (or $Y$) is a singleton set, we ease the notation and analogously define that the end $\omega_X \in X$ (or $X$) is separated from $Y$ (or from the end $\omega_Y \in Y$).

Similarly, we define for $W \subseteq V(D)$ and a set of ends $Y$ of $D$ that a vertex set $S \subseteq V(D)$ \emph{separates} $W$ from~$Y$ in~$D$ if for every $\omega\in Y$ every $R\in\omega$ has a tail~$Q$ such that every $Q$--$W$ dipath meets~$S$.
As before, we ease the notation and make corresponding definitions in case $W$ or $Y$ are singleton sets.
Finally, we accordingly define how a set of vertices and ends is separated from a set of ends. 

For an end $\omega$ of~$D$, set
\[
\omega^-:=\{\eta <\omega\mid \eta \text{ end of }D,d^-(\eta)\geq 1\}.
\]

For a vertex $v\in V(D)$ and a ray $R$, we call an infinite family of $v$--$R$ dipaths an \emph{infinite $v$--$R$ fan} if they pairwise meet only in~$v$.

A vertex $v$ \emph{dominates} an end $\omega$ of a digraph $D$ if $\omega$ contains a ray and for every ray $R\in\omega$ there is an infinite $v$--$R$ fan and an $R$--$v$ dipath.
We denote by $\dom(\omega)$ the set of vertices dominating the end $\omega$.
Note that looking at infinitely many distinct tails of~$R$, the definition implies the existence of infinitely many distinct $R$--$v$ dipaths.
In contrary to the $v$--$R$ fan, these dipaths may pairwise intersect in more vertices than just~$v$.
For an arbitrary end $\omega$ of a digraph $D$, we define the \emph{combined in-degree} of~$\omega$, denoted by $\Delta^-(\omega)$, as
\[
d^-(\omega)+\inf\{|S|\mid
S\subseteq V(D)\text{ separates }\omega^-\cup \dom(\omega)\text{ from }\omega\}.
\]
Analogously, it is possible to define the \emph{combined out-degree} of~$\omega$.

In graphs the combined degree of an end is the maximum number of disjoint rays plus the number of vertices dominating that end.
This is known to be equal to the infimum over the sizes of the vertex sets in defining sequences of that end (see~\cite{Gollin_Heuer_combined_end_deg}).
For digraphs, the infimum over the sizes of the sets of exhausting sequences is not the same as the in-degree plus the number of dominating vertices of that end as the following example shows:

\begin{example}\label{ex:combDegree}
Let $D$ be the digraph as depicted in Figure~\ref{fig:example_end-deg}.
The in-degree of the end $\omega$ containing $R$ is~$1$ and there is no dominating vertex of that end.
Intuitively, the ray $R^-$ serves as a dominating vertex that is stretched out as a ray.
The combined in-degree of~$\omega$ is~$2$ as we can simply let $S$ consist of the bottom left vertex for the definition of the combined in-degree.
Furthermore, if $U_i$ is a set consisting of the $i$-th vertex of~$R$ together with its in-neighbour on~$R^-$, then $(U_i)_{i\in\mathbb N}$ is an exhausting sequence of~$\omega$ and there is no exhausting sequence of~$\omega$ with smaller limit inferior.

\begin{figure}[ht]
    \centering
\begin{tikzpicture}
    \draw[-{Latex[length=3mm]}] (4.7,0) -- (0,0);
    \draw[-{Latex[length=3mm]}] (4.7,0) -- (4,0);
    
    \foreach \x in {0,1,2,3,4}
        \fill (\x,0) circle (2pt);

    \draw (5.6,0) node[anchor=east] {$R^{\leftarrow}$};

    \draw[-{Latex[length=3mm]}] (0,1) -- (4.75,1);
    \foreach \x in {0,1,2,3,4}
        \fill (\x,1) circle (2pt);

    \draw (5.35,1) node[anchor=east] {$R$};

    \draw[-{Latex[length=3mm]}] (0,2) -- (4.75,2);
    \foreach \x in {0,1,2,3,4}
        \fill (\x,2) circle (2pt);

    \draw (5.6,2) node[anchor=east] {$R^{-}$};

    \draw[-{Latex[length=3mm]}] (0,0) -- (0,1);
    \draw[-{Latex[length=3mm]}] (2,0) -- (2,1);
    \draw[-{Latex[length=3mm]}] (4,0) -- (4,1);
    
    \draw[-{Latex[length=3mm]}] (1,1) -- (1,0);
    \draw[-{Latex[length=3mm]}] (3,1) -- (3,0);

    \draw[-{Latex[length=3mm]}] (0,2) -- (0,1);
    \draw[-{Latex[length=3mm]}] (1,2) -- (1,1);
    \draw[-{Latex[length=3mm]}] (2,2) -- (2,1);
    \draw[-{Latex[length=3mm]}] (3,2) -- (3,1);
    \draw[-{Latex[length=3mm]}] (4,2) -- (4,1);

    \draw[-{Latex[length=3mm]}] (0,0) arc[start angle=270, end angle=90, radius=1];

    \draw (5,0.5);
    \fill (5,0.5) circle (0.5pt);
    \draw (5.2,0.5);
    \fill (5.2,0.5) circle (0.5pt);
    \draw (5.4,0.5);
    \fill (5.4,0.5) circle (0.5pt);

    \draw (5,1.5);
    \fill (5,1.5) circle (0.5pt);
    \draw (5.2,1.5);
    \fill (5.2,1.5) circle (0.5pt);
    \draw (5.4,1.5);
    \fill (5.4,1.5) circle (0.5pt);    
\end{tikzpicture}
    \caption{The end containing $R$ has combined in-degree $2$, but contains no two disjoint rays and is not dominated by any vertex.}
    \label{fig:example_end-deg}
\end{figure}
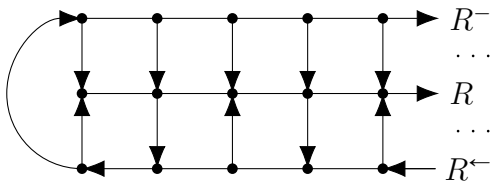

\end{example}

Our aim in the rest of this section is to show that the combined in-degree can be characterised via exhausting sequences as it is indicated by Example~\ref{ex:combDegree}.

\begin{lemma}\label{lem:weakExh}
    Let $D$ be a digraph and $\omega$ an end of~$D$ where $1 \leq d^-(\omega) < \infty$.
    Suppose there exists a finite $S \subseteq V(D)$ separating $\dom(\omega) \cup \omega^-$ from~$\omega$.
    Then there is a sequence $(U_i)_{i\in\mathbb N}$ with $U_i \subseteq V(D)$ and $|U_i|=d^-(\omega)$ for all $i\in\mathbb N$ such that every ray in $\omega$ meets some $U_i$ and if it meets $U_i$ and avoids $S$ then it also meets $U_{i+1}$.
\end{lemma}

\begin{proof}
    Let $U_1$ be a vertex set of size $d^-(\omega)$ such that $d^-(\omega)$ many disjoint rays $R_1,\ldots, R_{d^-(\omega)}$ with $R_i=x_0^ix_1^i\ldots$ for all $1\leq i\leq d^-(\omega)$ in~$\omega$ start at~$U_1$.
    Since $S$ is finite, we may assume that no $x_j^i$ lies in $S$ and that no $R_i$--$\dom(\omega)$ dipath exists in $D-S$ for any~$i$.
    Let us suppose that there exists no $V\subseteq V(D)$ of size $|U_1|$ that separates some ${X_k:=\{x_j^i\mid j\geq k \textnormal{ and } 1\leq i\leq d^-(\omega) \}}$ from $U_1$ in $D-S$.
    By Menger's theorem, there exist $|U_1|+1$ internally disjoint $U_1$--$X_k$ dipaths in $D-S$.
    Since these dipaths avoid~$\dom(\omega)$, there exist $d^-(\omega) + 1$ disjoint out-combs with teeth in $X_\ell$ for some~$\ell\in\mathbb N$ by a compactness argument.
    Thus, their spines must lie in ends $\mu\leq\omega$.
    But as all the dipaths used to construct the out-combs avoid~$S$, the spines must be in~$\omega$.
    Thus, $\omega$ contains more than $d^-(\omega)$ pairwise disjoint rays, which is impossible.
    So there exists $U_2\subseteq V(D)$ with $|U_2|=|U_1|$ that separates some $X_k$ with $k\in\mathbb N$ from $U_1$.
    We recursively define the other sets $U_\ell$ for $\ell>2$ in an analogous way, that is, such that $U_\ell\subseteq V(D)$ with $|U_\ell|=|U_1|$ separates some $X_k$ with $k\in\mathbb N$ from $U_{\ell-1}$.

    Now, let $R\in\omega$ be such that $R$ avoids $S$ and let us assume there is an $i\in\mathbb N$ such that $R$ contains a vertex $v$ from~$U_i$.
    Hence, $v \in V(R_j)$ for some $1 \leq j \leq d^-(\omega)$.
    Since $R\in\omega$, it contains infinitely many vertices from the rays $R_1,\ldots,R_{d^-(\omega)}$.
    This implies that $R$ contains vertices from $X_n$ for every $n \in\mathbb N$.
    As $U_{i+1}$ separates some $X_m$ from $U_i$ in $D-S$ and $R$ is disjoint from $S$, we know that $vR$ intersects $U_{i+1}$.    

    Let us now assume that $R\in\omega$ does not meet any~$U_i$.
    By the choice of $R_1,\ldots, R_{d^-(\omega)}$, the ray $R$ meets some~$R_i$.
    Let $Q$ be the ray $x_0^iR_ixR$, where $x$ is a vertex in~$R_i\cap R$.
    Then $Q$ meets any $U_j$ only in $x_0^iR_ix$ and thus it meets only finitely many $U_j$.
    Since it meets at least~$U_1$, this contradicts the property that we just proved.
    Thus, the sequence $(U_i)_{i\in\mathbb N}$ satisfies the claim.
\end{proof}

Lemma~\ref{lem:weakExh} shows that under its assumptions there exists an exhausting sequence all of whose elements have size $\Delta^-(\omega)$: simply take the sequence $(U_i\cup S)_{i\in\mathbb N}$.
While this just seems to be a part of a special case of Theorem~\ref{thm:main}, it will actually help us in the proof of that theorem.

\begin{lemma}\label{lem:sequence}
Let $D$ be a digraph and let $\omega$ be a countable end of~$D$ with $d^-(\omega)\geq 1$.
Let $(U_i)_{i\in\mathbb N}$ be an $\omega$-exhausting sequence such that $|U_i| \leq k$ for some $k \in \mathbb{N}$ and all $i\in\mathbb N$.
Then the following hold.
\begin{enumerate}[label = \rm (\roman*)]
\item\label{itm:sequence1} For every $\eta\in\omega^-$, either there exists a finite vertex set $S$ that separates $\eta$ from~$\omega$ and lies in all but finitely many $U_i$ or $(U_i)_{i\in\mathbb N}$ is $\eta$-exhausting.
\item\label{itm:sequence3} There exists a finite vertex set $S$ that separates $\dom(\omega)$ from~$\omega$ and lies in all but finitely many~$U_i$.
\end{enumerate}
\end{lemma}

\begin{proof}
Let $\eta\in\omega^-$.
Let us assume that there is no finite vertex set that separates $\eta$ from~$\omega$ and that lies in all but finitely many~$U_i$.
Let us suppose that there is a ray $Q\in \eta$ that avoids all vertex sets~$U_i$.
Let $S$ be the set of vertices that are contained in all but finitely many $U_i$.
Hence, $S$ is a finite set, and does not separate $\eta$ from~$\omega$ by assumption.
So there is a ray $R \in \omega$ such that for every tail $T$ of $R$ there is a ray $Q_T \in \eta$ starting at $T$ and avoiding~$S$.
By considering a tail, if necessary, and since $(U_i)_{i\in\mathbb N}$ is $\omega$-exhausting, we may assume that $R$ is disjoint from $S$ and the first vertex of~$R$ lies in~$U_n$ for some $n\geq 1$.
By assumption, we know that $Q_R \in \eta$ is a ray starting at a vertex $v$ on $R$ and avoiding~$S$.
Since $Q$ and $Q_R$ are equivalent, there exists a $Q_R$--$Q$ dipath $P_2$, starting at $q_R$ and ending at $q$, that is disjoint to $S$.
Then there exists $N\in\mathbb N$ with $N\geq n$ such that for all $j>N$ there is no common vertex of~$U_j$ and the ray $Q^* := RvQ_Rq_RP_2qQ$.
Because of $\eta\leq\omega$, there are infinitely many pairwise disjoint $qQ$--$R$ dipaths.
Thus, there exists one such dipath $P_3$ with first vertex $q^*$ and last vertex $r^* \neq r$ such that neither $P_3$ nor $r^*R$ contains any vertex from~$U_k$ for some $k>N$.
Then the ray $Q^*q^*P_3r^*R$ lies in~$\omega$, contains a vertex from $U_n$ but not from $U_k$.
This is a contradiction to $(U_i)_{i\in\mathbb N}$ being $\omega$-exhausting.
Thus, every ray in~$\eta$ meets some set~$U_i$.

Let us suppose that there exists a ray $Q\in\eta$ such that $Q$ contains a vertex from $U_i$ but not from $U_{i+1}$.
Let $x$ be on a ray $R\in\omega$ such that $xR$ avoids $U_{i+1}$, too, which is possible as $U_{i+1}$ is finite.
Since $Q\leq R$, there exists a $Q$--$xR$ dipath $P$ with starting vertex~$y$ and end vertex~$z$ that avoids $U_{i+1}$ such that $Qy$ meets $U_i$.
Then $QyPzR$ is a ray in~$\omega$ that contains a vertex from $U_i$ but not from $U_{i+1}$.
This is impossible since $(U_j)_{j\in\mathbb N}$ is $\omega$-exhausting.
Thus, \ref{itm:sequence1} follows.

Suppose that there is no finite vertex set separating $\dom(\omega)$ from~$\omega$ that lies in all but finitely many~$U_i$.
Let $S$ be the set of vertices that are contained in all but finitely many~$U_i$.
Hence, $S$ is a finite set and does not separate some $v \in \dom(\omega)$ from~$\omega$ by assumption.
Again, we may assume that there is a ray $R\in\omega$ that is disjoint from~$S$.
Similarly as in the proof of statement~\ref{itm:sequence1}, we can find a dipath $P$ starting at some $r \in R$ and ending in $v$ that avoids~$S$.
We may choose $r$ such that $Rr$ contains a vertex from some~$U_k$.
Then there exists $N\geq k$ such that $RrP$ contains a vertex $u\in U_N$ but $uRrP$ or $uP$, depending on whether $u$ lies on~$R$ or on~$P$, meets $U_M$ for $M\geq N$ at most in~$u$.
Furthermore, let $m > N$ such that $U_m \cap U_N \subseteq S$.
Then $uRrP$ or $uP$ is disjoint from~$U_m$.
Let $R'$ be a tail of~$R$ that is disjoint from~$U_m$.
Since there exists a $v$--$R'$ fan, there is a $v$--$R'$ dipath $P''$ with end vertex $r'$ on~$R'$ that avoids~$U_{m}$.
Then $uRrPvP'r'R'$ or $uPvP'r'R'$ is a ray in~$\omega$ that contains a vertex from~$U_N$ but avoids $U_m$, which contradicts that $(U_i)_{i\in\mathbb N}$ is $\omega$-exhausting.
This shows~\ref{itm:sequence3}.
\end{proof}

Now we are able to prove our main result.

\begin{thm}\label{thm:exhausting}
Let $D$ be a digraph and let $\omega$ be a countable end of~$D$ with $d^-(\omega) \geq 1$.
Then
\[
K(\omega):=\inf\left\{\liminf_{i\in\mathbb N}|U_i| \Bigm|(U_i)_{i\in\mathbb N} \text{ is an }\omega\text{-exhausting sequence}\right\}
\]
is the same as the combined in-degree $\Delta^-(\omega)$, where both values are considered in~$\mathbb N\cup\{\infty\}$.
\end{thm}

\begin{proof}
Let us define the following:
\begin{align*}
\delta^-(\omega):=\inf\Bigg\{|S|+\sum_{\eta\in B}d^-(\eta)\Biggm|\ &A\cup B=\omega^-\cup\{\omega\}, A\cap B=\emptyset, \omega\in B,\\
&S\subseteq V(D)\text{ separates }A\cup \dom(\omega)\text{ from }B\Bigg\}.
\end{align*}
We will actually prove that $K(\omega)=\delta^-(\omega)=\Delta^-(\omega)$, where we do not distinguish between distinct infinite cardinals.
Trivially, we have $\delta^-(\omega)\leq\Delta^-(\omega)$.

\medskip

In order to prove $K(\omega) \geq \delta^-(\omega)$, it suffices to prove the assertion for finite~$K(\omega)$.
Let $(U_i)_{i\in\mathbb N}$ be an $\omega$-exhausting sequence.
By thinning out the sequence $(U_i)_{i\in\mathbb N}$, we may assume without loss of generality that all sets $U_i$ have the same finite size.
Let $(A,B)$ be a partition of $\omega^-\cup\{\omega\}$ such that $B$ consists of those ends $\eta$ in $\omega^-\cup\{\omega\}$ for which $(U_i)_{i\in\mathbb N}$ is an exhausting sequence and such that the (finite) set of all those vertices that lies in all but finitely many $U_i$ does not separate $\eta$ from~$\omega$.
By Lemma~\ref{lem:sequence}\,\ref{itm:sequence3}, there exists a finite vertex set $S_\dom$ that separates $\dom(\omega)$ from~$\omega$ and that lies in all but finitely many~$U_i$.
Since $(U_i)_{i\in\mathbb N}$ is an exhausting sequence for every element $\eta\in B$, there exists by Lemma~\ref{lem:sequence}\,\ref{itm:sequence1}, for every $\mu\in A\cap \eta^-$, a finite vertex set $S_{\mu,\eta}$ that separates $\mu$ from~$\eta$ and lies in all but finitely many~$U_i$.
We claim that for every $\mu\in A \setminus \eta^-$ the set $S_{\mu,\omega}$ separates $\mu$ from~$\eta$ and lies in all but finitely many~$U_i$.
If $S_{\mu,\omega}$ does not separate $\mu$ from~$\eta$, then $S_{\mu,\omega}$ does not separate $\mu$ from~$\omega$, as $S_{\mu,\omega}$ is not separating $\eta$ from $\omega$ by definition of~$B$, a contradiction.
As~$K(\omega)$ is finite, there exists a finite vertex set~$S$ that separates~$A$ from~$B$ and that lies in all but finitely many~$U_i$.

If $B$ is infinite or has an element of infinite in-degree, then there exist more than $K(\omega)$ pairwise disjoint rays in elements of~$B$.
Since $(U_i)_{i\in\mathbb N}$ is $\omega$-exhausting, there exists an $N\in\mathbb N$ such that for all $i\geq N$ the set $U_i$ contains vertices from more than~$K(\omega)$ many of these rays.
This contradicts the definition of~$K(\omega)$.
Hence, $B$ is finite and every element of~$B$ has finite in-degree.
Thus, the maximum number of pairwise disjoint rays in elements of~$B$ is finite and is the same as $\sum_{\eta\in B}d^-(\eta)$.
Since $S\cup S_\dom$ is finite, we may assume that there are $\sum_{\eta\in B}d^-(\eta)$ many pairwise disjoint rays in elements of~$B$ each of which contains no vertex from $S\cup S_\dom$.
Since eventually all $U_i$ must contain a vertex from each of those rays, this completes the proof of $K(\omega)\geq \delta^-(\omega)$.

\medskip

In order to prove $K(\omega)\leq \delta^-(\omega)$, let us now assume that $\delta^-(\omega) < \infty$, i.\,e.\ there are a partition $(A,B)$ of~$\omega^-\cup\{\omega\}$ with $\omega\in B$ and some vertex set $S$ separating $A\cup\dom(\omega)$ from~$B$ such that $|S|+\sum_{\eta\in B}d^-(\eta)$ is finite.
In particular, $\delta^-(\omega)$ being finite implies that $B$ and~$S$ are finite.
Note that no end $\eta$ in~$B$ is separated from~$\omega$ by~$S$, as moving $\eta$ together with all ends in $B$ from which $\eta$ is not separated from~$B$ to~$A$ would make the value of~$\delta^-(\omega)$ smaller.
Let $\eta_1,\ldots,\eta_{|B|}$ be an enumeration of~$B$ such that $i\leq j$ if $\eta_i\leq \eta_j$.
Next we claim that, for every $i \in \{ 1, \ldots, |B| \}$, there is a finite vertex set $S_i$ that contains $S$ and separates all ends $\mu < \eta_i$ from $\eta_i$.
Since $B$ is finite, we can separate all ends $\mu < \eta_i$ with $\mu \in B$ with a finite vertex set $S_i'$ from $\eta_i$.
Hence, we can separate all ends $\mu < \eta_i$ with $\mu \in \omega^-$ from $\eta_i$ by $S \cup S_i'$.
Suppose there were a $\mu < \eta_i$ outside of $\omega^-$ which is not separated by $S \cup S_i'$ from $\eta_i$.
Then, we have $\mu \in \eta_i^-\subseteq \omega^-$, a contradiction.
Note that every vertex dominating an end $\eta_j < \omega$ also dominates $\omega$.
Therefore, every vertex dominating $\eta_i$ is separated from $\eta_i$ by~$S_i$ because $\eta_i \in B$ and $S \subseteq S_i$.
Thus, Lemma~\ref{lem:weakExh} implies that there are sequences $(U^j_i)_{i\in\mathbb N}$, for every $j \in \{ 1, \ldots, |B| \}$, with $|U^j_i|=d^-(\eta_j)$ for all $i\in\mathbb N$ such that every ray in~$\eta_j$ meets some $U^j_i$ and if it meets $U^j_i$ and avoids $S_j$ then it also meets $U^j_{i+1}$.
Let $R\in\omega$ be a ray.
Since $S$ does not separate any $\eta_j$ from~$\omega$, we may assume that there is a dipath from $R$ to every vertex in $U^j_i$ for every $1\leq j\leq |B|$ and every $i\in\mathbb N$.
For every $1\leq j\leq |B|$, let $\mathcal R^j$ be a maximal set of pairwise disjoint rays in~$\eta_j$ starting at~$U^j_1$.

\begin{claim}\label{clm:exhausting0}
    We may choose the $(U^j_i)_{i\in\mathbb N}$ such that there is a dipath from every tail of~$R$ to every vertex in $U^j_i$ that avoids~$S$ for every $1\leq j\leq |B|$ and every $i\in\mathbb N$ and such that the elements of all $\mathcal R^j$ for $1\leq j\leq |B|$ are disjoint from~$S$.
\end{claim}

\begin{claimproof}
    Since $S$ does not separate $\eta_j$ from~$\omega$, there exists a ray $Q \in \eta_j$ starting at a vertex on an arbitrary tail of~$R$ so that $Q$ is disjoint from~$S$.
    Hence, $Q$ intersects all but finitely many $U_i^j$.
    Let us now fix $i \in \mathbb N$ such that $Q$ intersects $U_i^j$, and let $x \in V(Q) \cap U_i^j$.
    Then there exists $N\geq i$ such that for every $i' \geq N$ and every $y\in U_{i'}^j$ there exists an $x$--$y$ dipath avoiding~$S$.
    Since $S$ is finite, the claim follows by taking a suitable subsequence of $(U^j_i)_{i\in\mathbb N}$ and thus taking the elements of $\mathcal R^j$ as suitable tails of the original ones.    
\end{claimproof}

\begin{claim}\label{clm:exhausting1}
    We may choose the sequences $(U^j_i)_{i\in\mathbb N}$ such that every ray in~$\eta_j$ that starts at $U_1^\ell$ with $\ell>j$ and avoids~$S$, and hence also every ray in~$\eta_j$ that starts at some $U^\ell_i$ with $\ell>j$ and avoids~$S$, must have a vertex from $\bigcup_{k\leq j}U_1^k$.
\end{claim}

\begin{claimproof}
    Let us suppose that the claim is false and that $j$ is smallest such that the claim fails for~$j$.
    In particular, the sets $U_1^k$ with $k<j$ are defined such that they have the desired property.
    Then there exists for infinitely many $i\geq 2$ a ray $Q_i \in \eta_j$ starting in some $U_1^\ell$ with $\ell>j$ that avoids $\bigcup_{k<j}U_1^k\cup U_i^j\cup S$, as otherwise we could replace the sequence $(U_i^j)_{i \in \mathbb{N}}$ by a suitable final subsequence.
    Since $Q_i\in\eta_j$, there exists a first vertex $x_i$ on~$Q_i$ that lies on some $Q\in\mathcal R^j$ such that $x_iQ$ does not meet $\bigcup_{k<j}U_1^k\cup U_i^j$.
    Note that $\{ x_i\mid i\in\mathbb N\}$ must be infinite.
    In the subdigraph induced by the starting dipaths $Q_ix_i$, there exists by Lemma~\ref{lem:StarComb} either a subdivided infinite out-star with leaves in elements of~$\mathcal R^j$ or an out-comb with teeth on elements of~$\mathcal R^j$.
    Let us first suppose that there exists a subdivided infinite out-star with centre $x$ and leaves in elements of~$\mathcal R^j$.
    Clearly, there exists an $x$--$R'$ fan for every ray $R' \in \eta_j$.
    And since $\eta_j \leq \omega$, there is also an $x$--$R''$ fan for every ray $R'' \in \omega$.
    By Claim~\ref{clm:exhausting0}, each $U_i^j$ can be reached from every tail of~$R$ via a dipath avoiding~$S$.
    Since all rays $Q_i$ avoid~$S$, we know that $x$ is reached from every tail of~$R$ via a dipath avoiding~$S$.
    So $x$ is dominating $\omega$, but not separated from it by~$S$, a contradiction to the choice of~$S$.
    
    Thus, we find an out-comb with spine~$T$ and teeth in $\{x_i\mid i\geq 2\}$ in that subdigraph.
    We may assume that $T$ has its first vertex in some $U_1^\ell$ with $\ell>j$.
    As $\mathcal R^j$ is finite, one of its elements contains infinitely many teeth of the out-comb.
    So~$T$ must lie in an element of $\omega^-\cup\{\omega\}$, but as all rays $Q_i$ avoid~$S$, it cannot lie in~$A$, so it lies in some $\eta_k \leq \eta_j$.
    By the minimal choice of~$j$, if $k<j$, then $T$ meets $\bigcup_{i<k}U_1^i$, which is impossible as all the dipaths $Q_ix_i$ avoid that set.
    Hence, we have $k=j$.
    By the choice of $(U^j_i)_{i\in\mathbb N}$, $T$ must contain a vertex from some~$U^j_m$, and thus from every $U^j_i$ for $i\geq m$ as the sequence was chosen according to Lemma~\ref{lem:weakExh}.
    This shows that infinitely many $Q_i$ contain a vertex from all $U^j_i$ for $i\geq m$, which contradicts their choices.
    Thus, there exists $n\in\mathbb N$ such that all rays in $\eta_j$ starting at some $U^\ell_1$ with $\ell>j$ that avoid $\bigcup_{k<j}U_1^k\cup S$ contain a vertex from~$U^j_n$.
    Removing all elements before $U^j_n$ from the sequence $(U^j_i)_{i\in\mathbb N}$ leaves with a sequence for~$j$ as desired.
\end{claimproof}

Note that Claim~\ref{clm:exhausting1} implies that $S\cup\bigcup_{j<|B|}U^j_1$ separates $\omega^-\cup \dom(\omega)$ from~$\omega$.
Furthermore, Claim~\ref{clm:exhausting1} implies that we may have chosen $S_i=S\cup\bigcup_{j<i}U^j_1$, which we assume for the rest of the proof.

We will define a sequence $(V_i)_{i\in\mathbb N}$ that is $\omega$-exhausting.
We set
\[
V_1:=S\cup \bigcup_{1\leq j\leq |B|}U_1^j
\]
and
\[
V_i:=S\cup\bigcup_{1\leq j< |B|}U_1^j\cup U^{|B|}_i.
\]

Note that
\[
V_i=V_1\sm U_1^{|B|}\cup U^{|B|}_i=S_{|B|}\cup U^{|B|}_i.
\]
Let $Q$ be a ray in~$\omega$.
By definition of $(U_i^{|B|})_{i\in\mathbb N}$, the ray~$Q$ contains a vertex from some~$U_i^{|B|}$.
If it also contains a vertex of~$S_{|B|}$, then the ray contains a vertex from all $V_j$ for $j\geq i$ by their definition.
If $Q$ contains no vertex from $S_{|B|}$, then the definition of the $(U_k^{|B|})_{k\in\mathbb N}$
implies that $Q$ contains a vertex from $U^{|B|}_{i+1}$ and inductively from all $U^{|B|}_j$ for $j\geq i$.
Thus, the sequence $(V_i)_{i\in\mathbb N}$ is $\omega$-exhausting.
This implies $K(\omega)\leq\delta^-(\omega)$.

\medskip

Since $\delta^-(\omega)\leq K(\omega)$ and since the sequence $(V_i)_{i\in\mathbb N}$ that we constructed in the proof of $K(\omega)\leq\delta^-(\omega)$ is $\omega$-exhausting and has the property $|V_i|\leq\delta^-(\omega)$ for all $i\in\mathbb N$, we have $|V_i|\leq K(\omega)$.
Since all $V_i$ have the same size, we have $|V_i|=K(\omega)$ for all $i\in\mathbb N$.
As $S_{|B|}$ separates $\omega^-\cup \dom(\omega)$ from~$\omega$, this implies
\[
\Delta^-(\omega)\leq |S_{|B|}|+d^-(\omega)=|V_1|=K(\omega).
\]

Thus, we have proved
\[
\delta^-(\omega)\leq\Delta^-(\omega)\leq K(\omega)\leq \delta^-(\omega),
\]
which completes the proof.
\end{proof}

\section{Edge-degrees of ends}
\label{sec:edge deg}

In this section, we will consider edge-disjoint rays in ends of digraphs and discuss the corresponding structural results for ends containing an arbitrary or an infinite number of pairwise edge-disjoint rays.
The proofs of the results are essentially the same as those given in Section~\ref{sec:deg} and Section~\ref{sec:ex}, which is why we omit them here.
The first result corresponds to Theorem~\ref{thm:endDeg} for edge-disjoint rays.

\begin{thm}\label{thm:EdgeDegree}
Let $D$ be a digraph.
\begin{enumerate}[label = {\rm (\roman*)}]
\item If an end of~$D$ contains $n$ pairwise edge-disjoint rays for all $n\in\mathbb N$, then it contains infinitely many pairwise edge-disjoint rays.
\item If an end of~$D$ contains $n$ pairwise edge-disjoint anti-rays for all $n\in\mathbb N$, then it contains infinitely many pairwise edge-disjoint anti-rays.\qed
\end{enumerate}
\end{thm}

Theorem \ref{thm:EdgeDegree} enables us to define the \emph{edge-in-degree} (the \emph{edge-out-degree}) of an end of a digraph as the maximum number of pairwise edge-disjoint rays (anti-rays) in that end.

Similarly as for vertex-disjoint rays and anti-rays as in Theorem~\ref{thm:InfRaysInfAntiRays}, we obtain a digraph that has infinitely many pairwise edge-disjoint rays and infinitely many pairwise edge-disjoint anti-rays such that every ray and every anti-ray share an edge.

\begin{thm}\label{thm:ExEdgeDisj}
There exists a digraph $D$ with the following properties:
\begin{enumerate}[label = \rm (\roman*)]
    \item\label{itm:ExEdgeDisj1} $D$ contains infinitely many pairwise edge-disjoint rays.
    \item\label{itm:ExEdgeDisj2} $D$ contains infinitely many pairwise edge-disjoint anti-rays.
    \item\label{itm:ExEdgeDisj3} Every ray and every anti-ray of~$D$ share an edge.
\end{enumerate}
\end{thm}

\begin{proof}
    Let $D'$ be the digraph constructed in the proof of Theorem~\ref{thm:InfRaysInfAntiRays}.
    We replace every vertex $u$ by two vertices $u^-$ and $u^+$ and every edge $uv$ by the edge $u^+v^-$.
    We add also all edges of the form $u^-u^+$ for $u\in V(D')$ in order to obtain the digraph $D$.
    Obviously, $D$ satisfies \ref{itm:ExEdgeDisj1} and \ref{itm:ExEdgeDisj2}.
    Since every vertex of~$D$ has either a unique out-neighbour or a unique in-neighbour, edge-disjoint dipaths in~$D$ induce disjoint dipaths in~$D'$.
    Thus, $D$ must satisfy~\ref{itm:ExEdgeDisj3}.
\end{proof}

\section*{Acknowledgement}

We thank Florian Reich for discussions that led to a simpler definition of the combined in-degree of an end.

\bibliographystyle{amsplain}
\bibliography{degree}

\end{document}